\documentclass[review,compress]{elsarticle}
\usepackage{lineno}
\usepackage[colorlinks,linkcolor=black]{hyperref}

\usepackage{amsmath,amsthm,amssymb}
\usepackage{multirow}
\usepackage{rotating}
\usepackage{algorithm}
\usepackage{algcompatible}
\usepackage{graphicx} 
\usepackage{epstopdf}
\usepackage{color}
\usepackage{ulem}

\modulolinenumbers[5]

\journal{XXX}
\newtheorem{theorem}{Theorem}[section]

\newtheorem{remark}{Remark}[section]

\allowdisplaybreaks[4]

\begin{document}

\begin{frontmatter}
\title{On computing HITS ExpertRank
via lumping the hub matrix
\footnote{This work was supported by National Natural Science Foundation of China (Grant Nos. 12001363, 71671125
)
.}}

\author[mymainaddress]{Yongxin Dong}
\author[mysecondaryaddress]{Yuehua Feng\corref{mycorrespondingauthor}}
\cortext[mycorrespondingauthor]{Corresponding author}
\ead{yhfeng@sues.edu.cn}
\author[mymainaddress]{Jianxin You}

\address[mymainaddress]
{\scriptsize School of Economics and Management,
Tongji University, Shanghai 200092, China
}
\address[mysecondaryaddress]
{\scriptsize School of Mathematics, Physics and Statistics,
Shanghai University of Engineering Science, Shanghai 201620, China}


\begin{abstract}
The dangling nodes is the nodes with no out-links in the web graph. It saves many computational cost and operations provided the dangling nodes are lumped into one node. In this paper, motivated by so many dangling nodes in web graph, we develop theoretical results for HITS by the lumping method. We mainly have three findings. First, the HITS model can be lumped although the matrix involved is not stochastic. Second, the hub vector of the nondangling nodes can be computed separately from dangling nodes, but not vice versa. Third, the authoritative vector of the nondangling nodes is difficult to compute separately from dangling nodes. Therefore, it is better to compute hub vector of the hub matrix in priority, not authoritative vector of the authoritative matrix or them simultaneous.

\end{abstract}

\begin{keyword}
HITS, Lumping, Nondangling nodes, Dangling nodes, Similarity transformation
\MSC[2020] 65F10, 65F50, 15A18, 15A21, 68P20


\end{keyword}

\end{frontmatter}

\section{Introduction}

The PageRank of Page and Brin, Hyperlink-Induced Topic Search (HITS) of Kleinberg, stochastic approach for link-structure analysis (SALSA) of Lempel and Moran use dominant eigenvector of non-negative matrices for ranking web page purpose \cite{Page1998StanfordDL,Kleinberg1999jacm,lempel2000stochastic}.
PageRank is used in Google, HITS is used in the ask.com search engine, and SALSA is a combination of PageRank and HITS \cite{lempel2000stochastic,langville2005survey}. Since the late 20th century, HITS is an another extremely successful modern web information retrieval application of dominant eigenvector. The resulting two ranking vectors (authority vector and hub vector) from HITS provide the ExpertRanks. The HITS method has broad applications such as product quality ranking, similarity ranking and so on. For discussions on HITS, together with the literature on modifications to overcome its weaknesses, we refer readers to \cite{Langville2006google}.

The eigenproblems related to web information retrieval and data mining can be of huge dimension. Because of memory constraint of computer, the power method has become the dominant method for solving the HITS and PageRank eigenproblems \cite{langville2005survey,Langville2006google}. As the web is very enormous, the web may contain more than $10^9$ pages and increases quickly and dynamically. It can even take much time (several hours or days) to compute a large web ranking vector. For the search-dependent HITS, the computation problem which only involves users' query related nodes is relatively small, that is why there is relatively less acceleration work on HITS. While for the search independent HITS, the matrices involved are usually of tremendous dimension and effective numerical accelerations are very desirable \cite{zhou2012jcam}.

As we know that the power method will be lack of efficiency when the eigen gap between the largest eigenvalue ($\lambda_1$) and the second largest eigenvalue ($\lambda_2$ in magnitude), is close enough to 1. The famous Krylov subspace methods can converge faster than the power method, but they are not suitable for such web information problem due to relatively large storage and subspace dimension and so on \cite{langville2005survey,Langville2006google,elden2007matrix}. Therefore, many acceleration methods for information retrieval model calculations include the aggregation methods \cite{Langville2006google}, extrapolation methods \cite{Brezinski2006simax,FengDong2021bims}, and two stage acceleration methods \cite{Lee2007im,Dong2017Calcolo} and other contributions \cite{langville2005survey,elden2007matrix} are developed. Among them, most methods consider the difficult computation case that the gap-ratio $\frac{|\lambda_2|}{|\lambda_1|}$ approximates 1. By lumping Google matrix, Ipsen and Selee analyzed the relationship between rankings of nondangling nodes and rankings of dangling nodes \cite{Ipsen2007simax}. To improve the computational efficiency of HITS ExpertRank, the filtered power method combining Chebyshev polynomials is proposed\cite{zhou2012jcam}. For more theoretical and numerical results on the web information retrieval model are available, see \cite{langville2005survey,Langville2006google,elden2007matrix,
FengDong2021bims,dong2021comments}. One may raise a question that whether the HITS model can still have similar lumping results, thus the computation cost can be reduced even if the matrix involved is not stochastic?

The rest of this paper is organized as follows. Section \ref{sec:HITSModel} describes the HITS model briefly. Section \ref{sec:Lumping} derives the main approach and theorems for lumping the HITS model. Finally, we give a short conclusion in section \ref{sec:conclusion}.

\section{HITS model}\label{sec:HITSModel}
The HITS model uses an adjacency matrix $L\in\mathbb{R}^{n\times n}$ to describe web link structure graph. The eigenvetor of $L^TL$ or $LL^T$ are employed to reveal the relative importance (rank) of corresponding web pages' authoritative vectors or hub vectors. Kleinberg \cite{Page1998StanfordDL} invented new matrices defined by
\begin{equation}\label{equ:authhubmatrix}
L^TL\in \mathbb{R}^{n\times n}, \quad LL^T \in \mathbb{R}^{n\times n},
\end{equation}
respectively, where $L$ is an adjacent matrix given by
\begin{equation*}
L_{ij} = \left\{ \begin{array}{l}
 1,~~ \mbox{if page $i$ links to page $j$,}\\
 0,~~\rm{otherwise}. \\
 \end{array} \right.
 \end{equation*}
In \eqref{equ:authhubmatrix}, if web page $i$ has no outlinks (i.e., image files, pdf with on links to other pages), it is called a dangling node; otherwise it is called a nondangling node.

The HITS method updates $v_a$ and $v_h$ iteratively from some initial vectors  $v_a^{(0)}$ and $v_h^{(0)}$,
\begin{align}\label{equ:vavhkk-1}
v_a^{(k)}=L^Tv_h^{(k-1)},\quad v_h^{(k)}=L v_a^{(k)},\quad k=1,2,\cdots .
\end{align}
Once one of $v_a$ and $v_h$ is convergent, the other vector is solved by multiplying $L$ or $L^T$. From \eqref{equ:vavhkk-1}, we have the following expression. The authoritative vector $\pi_a$ of the authoritative matrix $L^TL$ and the hub vector $\pi_h$ of the hub matrix are defined by computing the principle eigenvector of $L^TL$ and $LL^T$
\begin{align}\label{equ:LTLvecLLTvec}
\lambda_{\max}\pi_a = L^TL \pi_a,~\lambda_{\max}\pi_h=LL^T\pi_h,~\mbox{where}~ \pi_a\geq 0,~ \pi_h\geq 0,~ \| \pi_a \|_1 = 1,~\|\pi_h\|=1,
\end{align}
respectively. The matrix $L^TL$ or $LL^T$ is a symmetric positive semi-definite matrix, thus it has nonnegative eigenvalues. ExpertRanks is provided by the authoritative and hub vectors from HITS.

The expression \eqref{equ:LTLvecLLTvec} can't guarantee the uniqueness of $\lambda_{\max}$, $\pi_a$ or $\pi_h$. To make the uniqueness, we modify $L^TL$ and $LL^T$ such that they are primitive matrices. The authoritative matrix $A$ and hub matrix $H$ are defined by
\begin{align}
A=\xi L^TL+\frac{(1-\xi)}{n}ee^T,~ H=\xi LL^T+\frac{(1-\xi)}{n}ee^T,~ \mbox {where}~0<\xi <1,
~ e=\begin{bmatrix}1&\cdots&1\end{bmatrix}^T,
\end{align}
respectively. Accordingly, the authoritative and hub vector are defined by
\begin{align}\label{equ:mainpiapih}
\pi_a^TA=\pi_a^T, \quad \pi_h^TH=\pi_h^T,\quad \mbox{where}\quad \pi_a\geq 0,~\pi_h\geq0,~\pi_a^Te=1,~\pi_h^Te=1,
\end{align}
respective, where $e$ is a suitable length vector of all ones. In this paper, we thus mainly discuss the computation problem \eqref{equ:mainpiapih}. In the following section, we try to make theoretical and practical contributions for computing purposes.

\section{Lumping and related theorems}\label{sec:Lumping}

The adjacency matrix is lumpable if all dangling nodes are lumped into a single node \cite{Lee2007im,Ipsen2007simax}. According to \eqref{equ:authhubmatrix}, the adjacency matrix admits the structure
\begin{align}\label{equ:chooseP}
PLP^T=
\begin{bmatrix}
L_{11}&L_{12}\\
0&0
\end{bmatrix}\quad \mbox{and}\quad PL^TP^T=
\begin{bmatrix}
L_{11}^T&0\\
L_{12}^T&0
\end{bmatrix},
\end{align}
where $P$ is a suitable permutation matrix, $L_{11}\in\mathbb{R}^{k\times k}$, $L_{12}\in\mathbb{R}^{k\times(n-k)}$ and $k$ is the number of nondangling nodes.
Then we have \begin{align}\label{equ:PLPTPLTPT}
PLL^TP^T=PLP^T PL^TP^T
=&
\begin{bmatrix}L_{11}&L_{12}\\
0&0
\end{bmatrix}\begin{bmatrix}
L_{11}^T&0\\
L_{12}^T&0
\end{bmatrix} \nonumber\\
=&
\begin{bmatrix}
L_{11}L_{11}^T+L_{12}L_{12}^T&0\\
0&0
\end{bmatrix}.
\end{align}
After the primitive modification, we obtain
\begin{align}
PHP^T
=&P\left(\xi LL^T+\frac{1-\xi}{n}ee^T\right)P^T\nonumber\\
=&\xi \begin{bmatrix}
L_{11}L_{11}^T+L_{12}L_{12}^T&0\\
0&0
\end{bmatrix}+\frac{1-\xi}{n}ee^T\nonumber\\
=&
\begin{bmatrix}
\xi (L_{11}L_{11}^T+L_{12}L_{12}^T)+\frac{1-\xi}{n}ee^T&\frac{1-\xi}{n}ee^T\\
\frac{1-\xi}{n}ee^T&\frac{1-\xi}{n}ee^T
\end{bmatrix},
\end{align}
where $e$ is a suitable length vector of all ones. By using the lumping method and the similarity transformation matrix
\begin{align*}
Y=I_{n-k}-\widehat{e}e_1^T,\quad\mbox{with}\quad \widehat{e}=e-e_1=\begin{bmatrix}0,1,\cdots,1\end{bmatrix}^T\in \mathbb{R}^{n-k},
\end{align*}
where $I_{n-k}=\begin{bmatrix}e_1\cdots e_{n-k}\end{bmatrix}$ denotes the identity matrix of order $n-k$, and $e_i(i=1,2,\cdots,n-k)$ is its $i-$th column vector \cite{dong2021comments}.
Define
\begin{align}
X=\begin{bmatrix}I_k&0\\
0&Y\end{bmatrix},
\end{align}
then we have
\begin{align}\label{equ:XPLLTPTX-1}
& XPHP^TX^{-1} \nonumber\\
=&
\begin{bmatrix}\xi (L_{11}L_{11}^T+L_{12}L_{12}^T)+\frac{1-\xi}{n}ee^T
&\frac{1-\xi}{n}ee^TY^{-1}\\
\frac{1-\xi}{n}Yee^T&\frac{1-\xi}{n}Yee^TY^{-1}
\end{bmatrix}\nonumber\\
=&\begin{bmatrix}
\xi (L_{11}L_{11}^T+L_{12}L_{12}^T)+\frac{1-\xi}{n}ee^T
&\frac{1-\xi}{n}ee^TY^{-1}e_1&\frac{1-\xi}{n}ee^TY^{-1}
\begin{bmatrix}e_2&\cdots&e_{n-k}\end{bmatrix}\\
\frac{1-\xi}{n}e_1^TYee^T&\frac{1-\xi}{n}e_1^TYee^TY^{-1}e_1
&\frac{1-\xi}{n}e_1^TYee^TY^{-1}\begin{bmatrix}e_2&\cdots&e_{n-k}\end{bmatrix}\\
\frac{1-\xi}{n}\begin{bmatrix}e_2^T\\ \vdots \\e_{n-k}^T\end{bmatrix}Yee^T
&\frac{1-\xi}{n}\begin{bmatrix}e_2^T\\ \vdots \\e_{n-k}^T\end{bmatrix}Yee^TY^{-1}e_1
&\frac{1-\xi}{n}\begin{bmatrix}e_2^T\\ \vdots \\e_{n-k}^T\end{bmatrix}Yee^T Y^{-1}\begin{bmatrix}e_2&\cdots&e_{n-k}\end{bmatrix}
\end{bmatrix}\nonumber\\
=&
\begin{bmatrix}
\xi (L_{11}L_{11}^T+L_{12}L_{12}^T)+\frac{1-\xi}{n}ee^T
&\frac{1-\xi}{n}ee^TY^{-1}e_1&\frac{1-\xi}{n}ee^TY^{-1}
\begin{bmatrix}e_2&\cdots&e_{n-k}\end{bmatrix}\\
\frac{1-\xi}{n}e_1^TYee^T&\frac{1-\xi}{n}e_1^TYee^TY^{-1}e_1
&\frac{1-\xi}{n}e_1^TYee^TY^{-1}\begin{bmatrix}e_2&\cdots&e_{n-k}\end{bmatrix}\\
0&0&0\\
\end{bmatrix}\nonumber\\
=&\begin{bmatrix}
\xi (L_{11}L_{11}^T+L_{12}L_{12}^T)+\frac{1-\xi}{n}ee^T
&\frac{1-\xi}{n}(n-k)e&\frac{1-\xi}{n}ee^TY^{-1}\begin{bmatrix}e_2&\cdots&e_{n-k}\end{bmatrix}\\
\frac{1-\xi}{n}e^T&\frac{1-\xi}{n}(n-k)
&\frac{1-\xi}{n}e^TY^{-1}\begin{bmatrix}e_2&\cdots&e_{n-k}\end{bmatrix}\\
0&0&0\\
\end{bmatrix},
\end{align}
where we have used the fact that $Y^{-1}e_1=e$.
Thus, we have proved the following theorem.

\begin{theorem}\label{thm:dong1}
With the above notation, let
\begin{equation}
X=\begin{bmatrix}
I_k& 0\\
0& Y
\end{bmatrix},\quad \mbox{where} \quad Y = I_{n-k} - \widehat{e}e_1^T \quad \mbox{and} \quad
\widehat{e} = e - e_1= \begin{bmatrix} 0 &1 &\cdots &1 \end{bmatrix}^T,
\end{equation}
and $P$ be a suitable permutation matrix. Then
\begin{align}\label{eqn:XPHPX-eqn:H1andH2}
XPHP^TX^{-1}= \begin{bmatrix}
H^{(1)}&H^{(2)}\\
0&0
\end{bmatrix},
\end{align}
where
\begin{align*}
H^{(1)} =\begin{bmatrix}
H^{(1)}_{11}    & \frac{1-\xi}{n}(n-k)e\\
\frac{1-\xi}{n}e^T   &\frac{1-\xi}{n}(n-k)\end{bmatrix}
~\mbox{and}~~
H^{(2)}=
\begin{bmatrix}
\frac{1-\xi}{n}ee^TY^{-1}\begin{bmatrix}e_2,&\cdots&e_{n-k}\end{bmatrix}\\
\frac{1-\xi}{n}e^T Y^{-1}\begin{bmatrix}e_2,&\cdots&e_{n-k}\end{bmatrix}
\end{bmatrix},
\end{align*}
with $H^{(1)}\in \mathbb{R}^{(k+1)\times (k+1)}$, $H^{(2)}\in \mathbb{R}^{(k+1)\times (n-k-1)}$ and $H^{(1)}_{11}=\xi (L_{11}L_{11}^T+L_{12}L_{12}^T)+\frac{1-\xi}{n}ee^T$.
$H^{(1)}$ has the same nonzero eigenvalues as $H$.
\end{theorem}
The following theorem distinguished the relationship between the hub ranking vector $\pi_h$ of $H$ and the stationary distribution $\sigma$ of $H^{(1)}$. As we can see, the leading $k$ elements of $\pi_h$ represent the hub vector due to the nondangling nodes, and the trailing $n-k$ elements stand for the hub vector associated with the dangling nodes. Thus the relationship between ranking of dangling nodes and that of nondanling nodes is derived. For ease of the following proof process, we show the structure of the submatrix $Y$. Separate the first leading row and column,
\begin{align}\label{equ:Y}
Y=\begin{bmatrix}1&0\\
-e&I_{n-k-1}\end{bmatrix}.
\end{align}
Motivated by validating the analytic relationship between ranking vector of nondangling nodes and that of dangling nodes in the hub matrix model, we present our main lumping results for HITS.

\begin{theorem}
\label{thm:dong2}
With the above notation, let 
\begin{align}\label{eqn:lumpedHatG(1)}
\sigma^T H^{(1)}=\lambda_{\max}\sigma^T, \quad
 \sigma \in \mathbb{R}^{k+1},\, \sigma \geq 0, \, \|\sigma\|_1 = 1,
\end{align}
where $H^{(1)}$ is defined by \eqref{eqn:XPHPX-eqn:H1andH2} and partition $\sigma^T = \begin{bmatrix} \sigma^T_{1:k} & \sigma_{k+1}\end{bmatrix}$,
where $\sigma_{k+1}$ is a scalar, and $\lambda_{\max}$ is the largest eigenvalue of $L^TL$ or $LL^T$. Then the hub vector of $H$ equals
\begin{equation}\label{eqn:pi_h^T}
\pi_h^T =
\begin{bmatrix} \sigma^T_{1:k}
& \frac{1}{\lambda_{\max}}\sigma^T
\begin{bmatrix}\frac{1-\xi}{n}ee^T \\ \frac{1-\xi}{n}e^T\end{bmatrix}\end{bmatrix} P,
\end{equation}
where $P$ is a suitable permutation matrix satisfying \eqref{equ:chooseP}.
\end{theorem}

\begin{proof}
According to Theorem \ref{thm:dong1}, the stochastic matrix $H^{(1)}$ of order $k+1$ has the same nonzero eigenvalues as $H$. From \eqref{eqn:XPHPX-eqn:H1andH2} and \eqref{eqn:lumpedHatG(1)}, we can obtain that $\left[\sigma^T\right.  \left. \frac{1}{\lambda_{\max}}\sigma^T H^{(2)} \right]$ is an eigenvector for $XPHP^T X^{-1}$ associated with the eigenvalue $\lambda_{\max}$. Therefore,
\begin{align}
\widetilde{\pi}^T=\begin{bmatrix}\sigma^T   &\frac{1}{\lambda_{\max}}\sigma^T H^{(2)} \end{bmatrix}XP
\end{align}
is an eigenvector of $H$ associated with $\lambda_{\max}$. Since $H$ and $H^{(1)}$ have the same nonzero eigenvalues, and the principle eigenvalue of $H$ is distinct, the stationary probability distribution $\sigma$ of $H^{(1)}$ is unique. We repartition
\begin{align}
\widetilde{\pi}^T=
\begin{bmatrix}
\sigma_{1:k}^T
&\begin{bmatrix}\sigma_{k+1}  &  \frac{1}{\lambda_{\max}}\sigma^T H^{(2)} \end{bmatrix}
\end{bmatrix}
\begin{bmatrix}I_k&0\\
0& Y
\end{bmatrix}P.
\end{align}
Multiplying out
$\widetilde{\pi}^T=\begin{bmatrix}\sigma_{1:k}^T
&\begin{bmatrix}\sigma_{k+1} & \frac{1}{\lambda_{\max}}\sigma^T H^{(2)} \end{bmatrix}
Y \end{bmatrix}P$.
Hence, by \eqref{eqn:XPHPX-eqn:H1andH2} and \eqref{equ:Y}, we have
\begin{align}
&\begin{bmatrix}
\sigma_{k+1} & \frac{1}{\lambda_{\max}}\sigma^T H^{(2)}
\end{bmatrix}
\begin{bmatrix}
1&0\\
-e&I_{n-k-1}
\end{bmatrix}\nonumber\\
=&\begin{bmatrix}
\sigma_{k+1} - \frac{1}{\lambda_{\max}}\sigma^T H^{(2)}e
& \frac{1}{\lambda_{\max}}\sigma^T H^{(2)}
\end{bmatrix} \nonumber\\
=&\begin{bmatrix}
\frac{1}{\lambda_{\max}}\sigma^T \begin{bmatrix}
 \frac{1-\xi}{n}(n-k)e\\
\frac{1-\xi}{n}(n-k)\end{bmatrix} - \frac{1}{\lambda_{\max}}\sigma^T H^{(2)}e
& \frac{1}{\lambda_{\max}}\sigma^T \begin{bmatrix}
\frac{1-\xi}{n}ee^TY^{-1}\begin{bmatrix}e_2&\cdots&e_{n-k}\end{bmatrix}\\
\frac{1-\xi}{n}e^TY^{-1}\begin{bmatrix}e_2&\cdots&e_{n-k}\end{bmatrix}
\end{bmatrix}
\end{bmatrix} \nonumber\\
=&\begin{bmatrix}
\frac{1}{\lambda_{\max}}\sigma^T \begin{bmatrix}
 \frac{1-\xi}{n}e(e^Te)\\
\frac{1-\xi}{n}(e^Te)\end{bmatrix} - \frac{1}{\lambda_{\max}}\sigma^T \begin{bmatrix}
\frac{1-\xi}{n}ee^T\widehat{e}\\
\frac{1-\xi}{n}e^T\widehat{e}
\end{bmatrix}
& \frac{1}{\lambda_{\max}}\sigma^T
\begin{bmatrix}
\frac{1-\xi}{n}ee^T\\
\frac{1-\xi}{n}e^T
\end{bmatrix}\begin{bmatrix}e_2&\ldots&e_{n-k}\end{bmatrix}
\end{bmatrix} \nonumber\\
=&\begin{bmatrix}
\frac{1}{\lambda_{\max}}\sigma^T
\begin{bmatrix}
 \frac{1-\xi}{n}ee^T\\
\frac{1-\xi}{n} e^T\end{bmatrix}e_1
&\frac{1}{\lambda_{\max}}\sigma^T
\begin{bmatrix}
\frac{1-\xi}{n}ee^T\\
\frac{1-\xi}{n}e^T
\end{bmatrix}\begin{bmatrix}e_2&\ldots&e_{n-k}\end{bmatrix}\end{bmatrix}\nonumber\\
=&\frac{1}{\lambda_{\max}}\sigma^T
\begin{bmatrix} \frac{1-\xi}{n}ee^T\\ \frac{1-\xi}{n}e^T\end{bmatrix},
\end{align}
due to the fact that
\begin{align*}
&Y^{-1}\begin{bmatrix}e_2&\cdots&e_{n-k}\end{bmatrix}=\begin{bmatrix}e_2&\cdots&e_{n-k}\end{bmatrix} \quad \mbox{and} \quad
\sigma_{k+1}=\frac{1}{\lambda_{\max}}\sigma^T
\begin{bmatrix}
\frac{1-\xi}{n}(n-k)e\\ \frac{1-\xi}{n}(n-k) \end{bmatrix}.
\end{align*}
Hence,
\begin{equation}
\widetilde{\pi}^T =
\begin{bmatrix} \sigma^T_{1:k}
& \frac{1}{\lambda_{\max}}\sigma^T
\begin{bmatrix}\frac{1-\xi}{n}ee^T \\ \frac{1-\xi}{n}e^T\end{bmatrix}
\end{bmatrix}P,
\end{equation}
where $P$ is a suitable permutation matrix satisfying \eqref{equ:chooseP}. As discussed above and $\pi_h$ is unique, we conclude that $\widetilde{\pi}=\pi_h$ if $e^T\widetilde{\pi}=1$.
\end{proof}
\begin{remark}
Theorem \ref{thm:dong2} shows that we can compute the ranking vector of submatrix $H^{(1)}$ which is derived from the hub matrix, and then recover the hub vector according to \eqref{eqn:pi_h^T}.
\end{remark}

\begin{remark}
One can generalize the concrete invertible matrix $Y$ in \eqref{equ:Y} in theorems \ref{thm:dong1} and \ref{thm:dong2} by any invertible similarity transformation matrix satisfying the condition $\widehat{Y}e=e_1$. For more detailed induction, we refer readers to \cite{dong2021comments}.
\end{remark}

\begin{remark}\label{re:3.3}
Since
\begin{align}\label{equ:trans-PLPTPLTPT}
PAP^T=PL^TLP^T = PL^TP^T PLP^T
=&\begin{bmatrix}
L_{11}^T&0\\
L_{12}^T&0
\end{bmatrix}
\begin{bmatrix}
L_{11}&L_{12}\\
0&0
\end{bmatrix} =
\begin{bmatrix}
L_{11}^T L_{11}&L_{11}^TL_{12}\\
L_{12}^TL_{11}&L_{12}^TL_{12}
\end{bmatrix},
\end{align}
we thus remark that it is cheaper to compute \eqref{equ:PLPTPLTPT}, rather than \eqref{equ:trans-PLPTPLTPT}. This is due to $PAP^T$ is denser than $PHP^T$. Hence, it is better to compute the hub matrix $LL^T$ first, when compared with the authoritative matrix $L^TL$, so does the hub and authoritative matrices' primitive modifications.
\end{remark}

\begin{remark}
Since one can permutate $L$ such that \eqref{equ:chooseP} holds, however, the web adjacency matrix is very sparse (usually there is only about ten entries per row), one can even permutate $L$ such that
$\widetilde{P}L\widetilde{P}^T=\begin{bmatrix}\widetilde{L}_{11}&0\\
\widetilde{L}_{21}&0\end{bmatrix}$, this phenomenon can be verified by web data matrix from on-line Florida sparse matrix collection\footnote{\href{w}{https://sparse.tamu.edu/}}. 
Particularly, the number of all zero columns may be more than the number of all zero rows. In this case, the authoritative ranking vector is recommended to have priority in computing. But note that this case is very rare.
\end{remark}

\section{Conclusion}\label{sec:conclusion}

In this paper, we have studied a HITS computation approach via lumping dangling nodes of hub matrix $LL^T$ into a single node. Thus we have answered the HITS computation question leaving in the introduction.

The HITS can be computed by lumping approach despite the involved matrices are not stochastic. The approach which we discuss is useful whether the HITS model is search-dependent or not. For the hub vector, Theorem \ref{thm:dong2} shows us that the rankings of nondangling nodes can be computed independently from that of dangling nodes; while rankings of dangling nodes depends on the rankings of nondangling nodes.
According to Remark \ref{re:3.3}, the authoritative vector is relatively difficult to compute when compared with the hub vector. Thus we suggest that it is better to compute the hub vector in priority rather than authoritative vector or both of them simultaneously.

Further researches may include how to compute SALSA by the lumping method.
Questions like the ranking vector relationship of dangling nodes and nondangling nodes in SALSA model are also worthy of studying.

\section*{References}
\bibliography{hitsLumping}

\end{document}